\documentclass[11pt]{amsart}
\usepackage{amsmath}
\usepackage{amssymb, amsthm}
\usepackage{latexsym}
\usepackage{palatino}
\usepackage{geometry}
\usepackage{color}

\usepackage[colorlinks, citecolor={blue}]{hyperref}

\usepackage{pgfplots}

\newtheorem{thm}{Theorem}[section]
\newtheorem{prop}[thm]{Proposition}
\newtheorem{lemma}[thm]{Lemma}

\newtheorem{conj}[thm]{Conjecture}

\newtheorem{propmain}{Proposition}
\newtheorem{main}[propmain]{Theorem}
\newtheorem{conjmain}[propmain]{Conjecture}

\theoremstyle{definition}
\newtheorem{remark}[thm]{Remark}
\newtheorem{invisimark}[thm]{}
\newtheorem{ex}[thm]{Example}

\numberwithin{equation}{section}
\numberwithin{table}{section}

\newcommand{\f}{\frac}
\renewcommand{\(}{\left(}
\renewcommand{\)}{\right)}
\renewcommand{\~}{\widetilde}
\newcommand{\bmx}{\begin{pmatrix}}
\newcommand{\emx}{\end{pmatrix}}
\newcommand{\leg}[2]{ { \bmx #1 \\ #2 \emx } }
\newcommand{\ssk}[2]{ { \Big \{ \begin{matrix} #1 \\ #2 \end{matrix} \Big \} } }

\newcommand{\opt}{\mathrm{opt}}

\renewcommand{\geq}{\geqslant}
\renewcommand{\leq}{\leqslant}

\renewcommand{\ge}{\geqslant}
\renewcommand{\le}{\leqslant}

\title{How often should you clean your room?}
\author[K.\ Martin and K.\ Shankar]{Kimball Martin$^\ast$ and Krishnan Shankar$^\dagger$}

\thanks{$\ast$ Supported in part by a Simons grant. $\dagger$ Supported in part by an NSF grant.}

\address{Department of Mathematics, University of Oklahoma, Norman, 
             OK 73019}
\email{kmartin@math.ou.edu}
\email{shankar@math.ou.edu}

\begin{document}

\newcommand{\spacing}[1]{\renewcommand{\baselinestretch}{#1}\large\normalsize}
\spacing{1.165}

\begin{abstract} We introduce and study a combinatorial optimization problem motivated by the
question in the title.  In the simple case where you use all objects in your room equally
often, we investigate asymptotics of the optimal time to clean up in terms of the number
of objects in your room.  In particular, we prove a logarithmic upper bound, solve an approximate
version of this problem, and conjecture a precise logarithmic asymptotic.
\end{abstract}

\maketitle

\thispagestyle{empty}

\section*{Introduction}

Suppose you have $n$ objects in your room which are totally ordered.  For simplicity,
let us say they are books on shelves alphabetized by author and title.  If you are looking for
a book (assume you remember the author and title, but not its location on the shelves),
the most efficient algorithm is a binary search.  Namely, look at the book in the  middle of the 
shelf, and because of the ordering, now you can narrow your search by half.  Repeat this process
of halving your search list, and you can find your book in about $\log_2 n$ steps.
(Here is perhaps a better model for humans naturally search:  go to where you think the book should
be, scan that area, and if need be jump to a different area based on the ordering.  
However a logarithmic cost still seems like a good model for this process.)

The theory of searching (and sorting) algorithms is of course well studied in computer 
science---what is not, however, is what happens after that for humans.  Namely, after you are done
with your book, you can do one of two things: either put it back on the shelf, which we will also say
takes about $\log_2 n$ time, or leave it on your desk, which takes no time.  The latter is of course
more efficient now, but if you keep doing this, eventually all of your books will wind up as
an unsorted pile on your desk.  Then when you search for a book, you essentially have to go 
through your pile book by book (a sequential, or linear, search), which takes about $\frac n2$
time, and thus is not very efficient for $n$ large.  

The question we are interested in here is: when is the optimal time to clean up?  That is,
over the long run, what is the optimal value $m_\opt(n)$ of $m$ ($1 \le m \le n$) at which you should put
all the books in the pile back, in order, on the shelf, in the sense that the the average search
plus cleanup cost (per search) is minimized.  Here we assume the cleanup algorithm is to simply go through
the pile, book by book, and find the right location for each book on the shelves via a binary search
(see Remark \ref{alt-cleanup-rem} for a discussion of other cleanup algorithms).

The paper is organized as follows.  (See Section \ref{overview-sec} for a more detailed overview.)
In Section \ref{intro-sec}, after first formulating this problem precisely, we 
will discuss four different models and focus on the (generally unreasonable) case of the uniform distribution,
 i.e., where you use all objects in your room equally often.  It might be more realistic to consider a power law distribution, but even the simple case of the uniform distribution is not so easy.  The different models
 correspond to having either complete or no memory of what is in the pile, and having numbered
shelves (each object has a designated location on the shelves) or unnumbered shelves 
(only the relative order of books is important).  
 
In Section \ref{sec2}, we analyze the search and cleanup cost functions in some detail for each of these models.
Our first result is that, in each of these models, one should not clean up immediately (see Proposition \ref{propmain1} below).  In fact, if $n$ is small enough, one should never cleanup (see 
Remarks \ref{never-cleanup-rem} and \ref{M3-calc-rem}).
In Section \ref{M4-sec}, we restrict ourselves to complete memory with numbered shelves for
simplicity, and prove that one should clean up before about $4\log_2(n)$ objects are in the pile (see Proposition \ref{propmain2}).  A good lower bound for the $m_\opt(n)$ is not so easy, and so we instead
consider an approximate problem in Section \ref{approx-sec}.  
Based on the analysis from Section \ref{sec2}, we expect the optimal value
$\~m_\opt(n)$ of $m$ for the approximate problem to be a lower bound for $m_\opt(n)$.
We essentially determine exactly the optimal value of $m$ for the approximate problem
(Theorem \ref{main1}), which is about
$3\log_2(n)$, and then based on numerics conjecture that $m_\opt(n) \sim 3\log_2(n)$ 
(Conjecture \ref{mopt-asymp-conj-main}).  In fact we expect that for all four models with arbitrary distributions,
$m_\opt(n)$ grows no faster than $4\log_2(n)$.  Therefore, 
we humbly suggest you clean your room before $4 \log_2(n)$ objects are out.

Since we use a fair amount of (often similar looking) notation, we provide
a notation guide at the end for convenience (Appendix \ref{app}).
\medskip

{\bf Acknowledgements.}  It is a pleasure to thank our colleague Alex Grigo for helpful comments and suggestions.  We also thank our parents for never making us clean up our rooms.

\section{General Setup}\label{intro-sec}

\subsection{The Statement of the Problem} \label{problem-statement-sec} \hfill

We now make a general formulation of our problem, which we call a {\em search with cleanup optimization problem}.

\medskip
Let ${\bf X} = \{ X_1, \ldots, X_n \}$ be a finite set of distinct well-ordered objects, which we view as
a probability space with probability measure $\mu$.  
We consider the following discrete-time Markov chain, depending on a parameter $1 \le m \le n$.  

\begin{enumerate}
\item
At time $t=0$ each $X_i$ is in a sorted list\footnote{By list, we mean an indexed list, rather than a linked list.} $\mathcal L$, and there is an
unsorted pile $\mathcal P$ which is empty.  

\item
At any time $t \in \{ 0, 1, 2 , \ldots \}$, each $X \in \bf X$ is in exactly one of $\mathcal L$ and $\mathcal P$, i.e.,
$\bf X$ is a disjoint union $\bf X = \mathcal L \sqcup \mathcal P$.

\item
At any time $t \ge 1$ with $|\mathcal P| < m$, exactly one object $X=X_i$
is selected, and $X$ is selected with probability $\mu(X)$.  If the selected
 $X \in \mathcal P$, nothing changes.  Otherwise,
 then $X$ is removed from $\mathcal L$
 and added to $\mathcal P$.
 
 \item
At any time $t$, if $|\mathcal P| = m$, we stop the process.
\end{enumerate}

This process has finite stopping time with probability 1 provided at least $m$ elements of 
$\bf X$ have nonzero probabilities, which we will assume.
 
Note the state of the process at time $t$ is described simply by a subset $\mathcal P$ of $\bf X$, together
with a marked element $X_{i_t}$ (the object selected at time $t \ge 1$).
The set of possible states is then simply all subsets $\mathcal P$ of $\bf X$, together with a marked point
$X_{i_t}$, of cardinality at most $m$.

Associated to this process are two (nonnegative) cost functions, $S(X; \mathcal P)$ and 
$C(\mathcal P)$, which do not depend upon $t$.  Here $X \in \mathbf X$ and 
$\mathcal P \subset \mathbf X$.
The functions $S(X; \mathcal P)$ and 
$C(\mathcal P)$ are called the search and cleanup costs.

Let $\mathcal X_m = \mathcal X_{m,n}$ 
denote the set of finite sequences $\chi = ( X_{i_1}, \ldots, X_{i_\ell} )$ in $\bf X$  
such that (i) the underlying set $\{ X_{i_1}, \ldots, X_{i_\ell} \}$
 has cardinality $m$, and (ii) $X_{i_j} \ne X_{i_\ell}$ for $j < \ell$.  
We extend
the measure $\mu$ to a probability measure on $\mathcal X_m$ by
\begin{equation}
 \mu(\chi) = \prod_{j=1}^\ell \mu(X_{i_j}).
\end{equation}
Note the 
sequences $\chi \in \mathcal X_m$ are in 1-1 correspondence with the possible paths of
finite length for the Markov process
from the initial state up to the stopping state described in Step 4.  Namely, for $t = 0, \ldots, \ell$, 
let $\mathcal P_\chi(t)$ denote the set of elements $\{ X_{i_1}, \ldots, X_{i_t} \}$
(here $P_\chi(0) = \emptyset$).  Thus $\mathcal P_{\chi}(t)$
 represents the ``unmarked'' 
 state of the process from time $t=0$ until the stopping time $t=\ell$.
Furthermore each $\mu(\chi)$ is the probability of that path for the process.  

For example, suppose $n \ge 3$, $m=3$ and $\chi = (X_1, X_2, X_1, X_3)$.  This corresponds to selecting
$X_1$ at time 1, $X_2$ at time 2, $X_1$ at time 3, and $X_3$ at time 4, after which the process stops,
since we have selected $m=3$ distinct objects.
Specifically, at $t=1$ we have
$\mathcal P = \mathcal P_{\chi}(1) = \{ X_1 \};$
 at time $t=2$ we have 
$\mathcal P = \mathcal P_{\chi}(2) = \{ X_1, X_2 \};$
at time $t=3$, we have
$\mathcal P = \mathcal P_{\chi}(3) = \{ X_1, X_2 \};$
(unchanged); and at the stopping time $t=4$, we have
$\mathcal P = \mathcal P_{\chi}(4) = \{ X_1, X_2, X_3 \}.$
If $\mu$ is the uniform distribution on $\bf X$, the probability of this path is $\mu(\chi) = \f 1{n^4}$.


Given $\chi \in \mathcal X_m$, we let $\ell(\chi)$ be its length, i.e., the corresponding stopping time, and write $\chi = (\chi_1, \chi_2, \ldots, \chi_\ell)$ where $\ell = \ell(\chi)$.

Now we extend $S(X; \mathcal P)$ and $C(\mathcal P)$ to $\chi = (X_{i_1}, \ldots, X_{i_\ell})  \in \mathcal X_m$
 by
\begin{equation}
 S(\chi) = \sum_{j=1}^{\ell} S(X_{i_j}; \mathcal P_\chi(j-1) ) =  \sum_{j=1}^{\ell(\chi)} S(\chi_j; \mathcal P_\chi(j-1) )
\end{equation}
and
\begin{equation}
C(\chi) = C(\mathcal P_\chi(\ell(\chi)).
\end{equation}
These values are called the {\em total search} and {\em total cleanup costs} for the path $\chi$.

\medskip
We want to optimize the {\em average total cost function}
\begin{equation}
 F(m) : = F(m; n) = E \left[ \f{S(\chi) + C(\chi)}{\ell(\chi)} \right] = 
\sum_{\chi \in \mathcal X_m}\f{S(\chi) + C(\chi)}{\ell(\chi)}  \mu(\chi).
\end{equation}
Assume $\mu(X) \ne 0$ for each $X \in \bf X$.

\medskip
\noindent
{\bf Problem.} Given a {\em model} $\mathcal M = (\mathbf X, \mu, S, C)$, determine the value
$m_\opt(n) = m_\opt(n; \mathcal M)$ of $m \in \{ 1, 2, \ldots, n \}$ that minimizes $F(m; n)$.

\medskip
In the event that there is more than one such minimizing $m$---which we do not
typically expect---we may take, say, $m_\opt(n)$
to be the smallest such $m$, so that $m_\opt(n)$ is a well-defined function.

Here we will study the asymptotic behavior in the simple case of $\mu$
being a uniform distribution on $\mathbf X$ as $n=|\mathbf X| \to \infty$ for certain cost functions
$S$ and $C$ specified below.
We note the Markov process we consider arises in the coupon collector's (or birthday) problem, and
more generally, sequential occupancy problems (see, e.g., \cite{urns} or \cite{charalam}).  
The cost functions, however, make the analysis of this 
problem much more delicate than occupancy problems typically studied in the
literature.  It turns out that the expected value of the reciprocal of the waiting time in a sequential
occupancy problem plays a key role in our analysis.  Several results for the expected
value of the waiting time itself are known (e.g., see \cite{BB}), but not, to our knowledge, for its reciprocal.
\medskip

\subsection{Models and Cost Functions} \label{scf-sec} \hfill

From now on, we assume $\mu$ is the uniform distribution on $\bf X$ unless explicitly stated otherwise. There are four reasonable, simple search models to consider, all based on doing a binary search on 
$\mathcal L$
and a sequential search on $\mathcal P$.  Here we view $\mathcal P$ as an unordered set.  
The  models depend upon whether the positions of $\mathcal L$
(the ``shelves'') are numbered or not and whether the process is memoryless or not.
These models correspond to the following search algorithms 
$\bf A$ for an element $X$ of $\mathcal L \sqcup \mathcal P$.

For a memoryless process, at any time $t$, we assume we do not know what elements are in $\mathcal P$, i.e.,
we do not remember the state of the system.  Thus it is 
typically worthwhile to search $\mathcal L$ first, as searching $\mathcal L$ is much more efficient than
searching $\mathcal P$.  Hence for a memoryless process, we will always first search $\mathcal L$ for $X$.   
If this search is unsuccessful
(i.e., $X \in \mathcal P$), then we search $\mathcal P$.  

At the other extreme, one can consider the process where one has complete memory, i.e., at any time $t$, we know the state 
$\mathcal P$ of the system.  Thus if $X \in \mathcal L$, we simply search $\mathcal L$, and if 
$X \in \mathcal P$, we only search $\mathcal P$.

The other option in the model depends on the data structure for $\mathcal L$.  
Imagine $X_1, \ldots, X_n$ are books with a total ordering.  The $X_i$'s in $\mathcal L$
are the books that are ordered on bookshelves, whereas the $X_i$'s in $\mathcal P$ lie in an
unorganized pile on a desk.  If there is a marking on each shelf for the corresponding book, so each
book has a well defined position on the shelf, we say the shelves are numbered.  In this case, we
think of $\mathcal L$ as a list of size $n$ indexed by keys $k_1 < k_2 < \cdots < k_n$, where
$k_i$ points to  $X_i$ if $X_i \in \mathcal L$, and $k_i$ points to null if $X_i \in \mathcal P$,
and a search on $\mathcal L$, amounts to a search on $n$ keys, regardless of how many object
$X$ actually remain in $\mathcal L$.
Otherwise, the shelves are unnumbered, so only the relative position of the books on the shelves is important (akin to books shelved in a library stack). Here we simply view $\mathcal L$ as a sorted binary tree, and a search on $\mathcal L$ is really a search on the $|\mathcal L|$ objects in $\mathcal L$.

While shelf positions are not typically numbered for books, this situation of ``numbered shelves'' commonly
occurs in other situations, such as a collection of files each in their own labelled folder jacket.  Namely,
you may take out a file to look at, but leave the folder jacket in place so there is a placeholder for where the
file goes when you put it back.

With these models in mind, the four search algorithms $\mathbf A$ for an object $X$ in $\mathcal L
\sqcup \mathcal P$ can be described as follows.

\begin{itemize}
\item $\mathcal M_1$
(No memory, unnumbered shelves) $\bf A$: do a binary search on the $|\mathcal L|$ objects in $\mathcal L$;
if this fails, then do a sequential search on $\mathcal P$

\item $\mathcal M_2$ 
(No memory, numbered shelves) $\bf A$: do a binary search on the $n$ keys to find the correct position for
$X$ in $\mathcal L$; if it is not there, do a sequential search on $\mathcal P$

\item $\mathcal M_3$
(Complete memory, unnumbered shelves) $\bf A$: if $X \in \mathcal L$, do a binary search on the
$|\mathcal L|$ objects in $\mathcal L$; if $X \in \mathcal P$, do a sequential search on $\mathcal P$

\item $\mathcal M_4$
(Complete memory, numbered shelves) $\bf A$: if $X \in \mathcal L$, do a binary search on the
$n$ keys for $\mathcal L$; if $X \in \mathcal P$, do a sequential search on $\mathcal P$
\end{itemize}

Each of these algorithms naturally gives rise to a search cost function $S(X; \mathcal P)$ where
$X \in \mathcal L \sqcup \mathcal P$, namely the number of comparisons needed in this algorithm.
However, it is not necessary for us to write down these functions explicitly.  Rather, it suffices
to explicate the following average search cost functions. 
(In fact, one could replace the exact search cost $S(X; \mathcal P)$ by a certain average search cost
and be left with the same optimization problem---see Section \ref{nonunif-sec}.)

Let $s_{\mathcal L}(j)$ denote the average cost of a search for an object in $\mathcal L$
when $\mathcal L$ contains $n-j$ elements (we average over both the $n$ choose $n-j$ possibilities for 
$\mathcal L$ and the $n-j$ possibilities for the object).  Similarly, let $s_{\mathcal P}(j)$ denote the average
cost of a search for an object in $\mathcal P$ given $\mathcal P$ contains $j$ objects (again averaging over
all possibilities for $\mathcal P$ and the object).

We define the following average search cost functions for successful binary, failed binary and
sequential searches on $j$ objects:

\begin{align} \nonumber
b(j) &= \( 1 + \f 1j \) \log_2(j+1) - 1 \\
b_f(j) &= \log_2(j+1) \\ \nonumber
s(j) &= \f{j+1}2.
\end{align}

The formula for $s(j)$ is of course exactly the expected number of steps required for a successful 
sequential search.  It is easily
seen that when $j+1$ is a power of 2, $b(j)$ (resp.\ $b_f(j)$) is the exact expected number of steps
required for a successful (resp.\ failed) binary search on $j$ objects.  These functions are not
quite the exact average number of steps for binary searches for all $j$ (they are not generally rational),
 but as we are primarily interested in asymptotic behavior, we will work with the functions given above for simplicity. Note that $b(2^{r+1} -1) - b(2^r - 1) < 2$ and is in fact close to 1 for large $r$. So $b(n)$ in general is a reasonable
 approximation of the expected cost for a successful binary search.

Then, for the above four algorithms $\bf A$, the functions $s_{\mathcal L}(j)$ and
$s_{\mathcal P}(j)$ are given as follows.

\begin{itemize}

\item $\mathcal M_1$
(No memory, unnumbered shelves)
\begin{align}
s_{\mathcal L}(j) &= b(j) \\
s_{\mathcal P}(j) &= b_f(n-j) + s(j) \nonumber
\end{align}

\item $\mathcal M_2$
(No memory, numbered shelves)
\begin{align}
s_{\mathcal L}(j) &= b(n) \\
s_{\mathcal P}(j) &= b_f(n) + s(j) \nonumber
\end{align}

\item $\mathcal M_3$
(Complete memory, unnumbered shelves)
\begin{align}
s_{\mathcal L}(j) &= b(j) \\
s_{\mathcal P}(j) &= s(j) \nonumber
\end{align}

\item $\mathcal M_4$
(Complete memory, numbered shelves)
\begin{align} \label{M4-model}
s_{\mathcal L}(j) &= b(n) \\
s_{\mathcal P}(j) &= s(j) \nonumber
\end{align}

\end{itemize}

\begin{remark}
If $\mu$ were a highly skewed distribution, then it might be more 
efficient in the no memory models to do the pile search before a list search (see Section
\ref{nonunif-sec}).
\end{remark}

We now define our cleanup cost functions, based on the simple algorithm of doing a binary search for each
object in $\mathcal P$ to find the appropriate position to insert it into $\mathcal L$.  
(Even if one remembers the general area where the object should go, there is still the time needed to 
identify/arrange the exact spot and the time to physically place it there, and a logarithmic cost seems
like a reasonable model for this.) This
leads to two different possible cleanup cost functions, corresponding to the cases of numbered
and unnumbered shelves.

If the shelves are numbered, then the cleanup cost should just be the search cost to find the correct position for
each object in $\mathcal P$, and it makes sense to set
\[ C(\mathcal P) = \sum_{X \in \mathcal P} S(X; \emptyset), \]
where $S(X; \emptyset)$ denotes the search cost to find the position in $\mathcal L$ for $X$.
Note that there is no dependence upon what order we replace the objects.  However, we can make things a little
easier on ourselves if we wish.  Since we will just be considering an average of $C(\mathcal P)$ over $\chi$ (weighted
by $\f 1{\ell(\chi)}$), it will suffice to consider an average cleanup cost
\[ C_m = {\bmx n \\ m \emx}^{-1} \sum_{|\mathcal P| = m} C(\mathcal P). \]
Hence we have
\begin{equation} \label{Cmnum}
 C_m = m \cdot b(n), \qquad \mathcal M \in \{ \mathcal M_2, \mathcal M_4 \}.
\end{equation}

\medskip
If the shelves are unnumbered, then the cleanup cost in fact depends upon the order we replace the objects.  Let
us write $\mathcal P = \{ X_{i_1}, \ldots, X_{i_m} \}$ and suppose we place them back in order $X_{i_1}, \ldots, X_{i_m}$.
Write $S_{\mathcal L}(X)$ for the cost of a (failed if $X \not \in \mathcal L$) binary
search on $\mathcal L$ for the object $X$.
  Then the order-dependent cleanup cost is
\[ C_{od}(X_{i_1}, \ldots X_{i_m}) = S_{\mathcal L}(X_{i_1}) + S_{\mathcal L \cup \{ X_{i_1} \} }(X_{i_2})  
+ \cdots +  S_{\mathcal L \cup \{ X_{i_1}, \ldots, X_{i_{m-1}} \} }(X_{i_m} ). \]
Since $\mathcal P$ is unordered, we consider all cleanup orderings to occur with the same probability.
Hence it suffices to consider an average over all possible orderings:
\[ C(\mathcal P) = \f 1{m!} \sum  C_{od}(X_{i_1}, \ldots X_{i_m}), \]
where $(X_{i_1}, \ldots, X_{i_m})$ runs through all possible orderings of $\mathcal P$.

As before, since we will be taking an average of our cleanup costs over $\chi$ (weighted by $\f 1{\ell(\chi)}$), we can consider
the simpler quantities
\[ C_m = \f{1}{\bmx n \\ m \emx} \sum_{|\mathcal P| = m} C(\mathcal P), \]
as in the numbered case.
By additivity of the expected value, one sees
\begin{equation} \label{Cmunnum}
 C_m = \sum_{j=1}^m b_f(n-j), \qquad \mathcal M \in \{ \mathcal M_1, \mathcal M_3 \}.
\end{equation}

As with $S(X; \mathcal P)$, we could replace the exact cleanup cost $C(\mathcal P)$ with its average
over all subsets of size $\mathcal P$ (cf.\ \eqref{CP-nuf}).

\subsection*{Remarks}\hfill

\begin{invisimark} \label{alt-cleanup-rem}
This is not the only reasonable way to clean up.  One could first sort 
the objects in $\mathcal P$, which can be done in $O(m \log m)$ time, though the way humans
naturally sort is perhaps better modeled by insertion sort, which takes $O(m^2)$ time.  Then one can merge 
$\mathcal L$ and $\mathcal P$ in $O(n)$ steps, as in a linear merge sort.  This is more efficient than our above
algorithm if $m$ is relative large and one efficiently sorts $\mathcal P$.  Since our optimization problem is one in which $m$ should be
at most logarithmic in $n$ (cf.\ Proposition \ref{propmain2} and Remark \ref{model-comparison-rem}), our cleanup algorithm above is more efficient.
\end{invisimark}

\begin{invisimark}
Alternatively, one could do a binary-search-based merge sort after sorting $\mathcal P$ as follows.  
Say the ordering on 
$\mathbf X$ is $X_1 < X_2 < \cdots < X_n$.  Let $X_{j_1}, \ldots, X_{j_m}$
be the elements in $\mathcal P$ in sorted order, i.e., $j_1 < j_2 < \cdots < j_m$.  
First do a binary search to insert $X_{j_1}$ in $\mathcal L$.  Then do a binary search on 
$\mathcal L \cap \{ X_{j_1+1}, X_{j_1+2}, \ldots, X_n \}$ to find the position for $X_{j_2}$.  Continue
in this manner of binary searches on smaller and smaller subsets of $\mathcal L$, to replace all $m$
objects.  This may more be efficient than the cleanup algorithm we are using, depending on
how we sort $\mathcal P$ and the relative size of $m$ and $n$, and it may be interesting
to study our optimization problem with this type of algorithm.  However, it is only slightly more efficient when
$m$ is relatively small compared to $n$: suppose $m \approx \log n$ and one does
an insertion sort on $\mathcal P$; the insertion sort alone takes $O(\log^2 n)$ time, which is the same order as our original cleanup algorithm.  In light of the additional complications it brings,
we do not consider this type of cleanup here.
\end{invisimark}

\begin{invisimark}
One could also consider partial cleanups, where one does not put back all objects at the same time,
but only some of the items in $\mathcal P$.  We do not wish to consider such complications here.
Moreover, as it typically takes time and effort for humans to
switch between tasks, there seems to be extra efficiency in practice if one clean up all at once (or in
a few chunks), than in many small steps.
\end{invisimark}

\begin{invisimark} This model assumes all objects are in relatively close proximity, as in your room.  
If one wanted to consider a similar problem for objects in large library or warehouse, one should
include the cost of transit time for retrieving and putting back the objects
in the functions  $s_{\mathcal L}$ and $C(\mathcal P)$.  
The transit time should be
$O(\sqrt n)$ assuming the objects are organized in a 2-dimensional grid, or at least 3-dimensional
with bounded height.
\end{invisimark}

\subsection{Overview}\hfill \label{overview-sec}

Intuitively, there are three reasons why it may be better to wait to cleanup, i.e., why $m_\opt(n)$
might be greater than 1.  Assume $n$ is large.  

\medskip
(i)  If one has complete memory and there
are relatively few objects in the pile, the search cost for an object in the pile will be less than the
search cost for a random object in the list. 

(ii) If the shelves are not numbered and there
are relatively few objects in the pile, one will almost surely be searching for objects which are in
the sorted list, and this will go slightly faster if there are less than $n$ objects in the list.

(iii) In all four of the above models, the average cleanup cost per search should decrease as
$m$ increases.

\medskip
Thus in the case of complete memory, it is rather evident that we should have $m_\opt > 1$.
On the other hand, in the case of no memory, if one searches for  an object in the pile, one first
has to do a binary search on the list, which costs more than just searching for a random
element in the list.  So in the case of no memory, unnumbered shelves, it is not {\em a priori}
obvious whether this factor or points (ii) and (iii) will win out.
This is settled by our first result, which says one should never clean up immediately.

\begin{propmain} \label{propmain1}
Suppose $\mathcal M \in \{ \mathcal M_1, \mathcal M_2, \mathcal M_3, \mathcal M_4 \}$.  For any $n \ge 2$,
we have $m_\opt(n) > 1$.
\end{propmain}

This is not hard, and we provide two proofs: one by computing $F(1)$ and $F(2)$ explicitly in each 
model (see Section \ref{simp-calc-sec}), and another by observing $F(m) < F(1)$ whenver $m < 4b(n-m)$ (see Lemma \ref{F1comp-lem}).

In Section \ref{M4-sec}, we restrict ourselves for simplicity to the case of complete memory and
numbered shelves (model $\mathcal M_4$).  An upper bound for $m_\opt$ is not too difficult, since
after the pile is a certain size, each search will have an associated cost that is at least $F(1)$.
Specifically, we show

\begin{propmain} \label{propmain2}
Let $\mathcal M = \mathcal M_4$.  For $n \ge 1$, $m_\opt(n) < 4b(n) \le 4 \log_2(n+1)$.
\end{propmain}

The problem of obtaining a good lower bound seems much more difficult, and we use some
bounds shown in Section \ref{approx-sec} to construct an approximation $\~F(m)$ for $F(m)$
such that the (smallest if not unique)
 value $\~m_\opt(n)$ (see Section \ref{sec2} for the definition) of $m$ which minimizes $\~F(m)$ should satisfy 
$\~m_\opt(n) \le m_\opt(n)$ (Conjecture \ref{m4-conj1}).  While we can compute $\~m_\opt(n)$ for fairly large $n$ fairly 
quickly, the amount of time required to compute $m_\opt(n)$ is significant, so we can only
compare values of these functions for relatively small $n$ (see Table \ref{mopt-tab}),
but it appears that $\~m_\opt(n) \sim m_\opt(n)$.  Given that this is the case, one
would like to determine $\~m_\opt(n)$.

\begin{main}[Theorem \ref{approx-thm}] \label{main1}
For any $n\geq 5$, we have
\large
$$
3b(n) - \tfrac{3}{2} \leq \~m_\opt(n) < 3 b(n) + \tfrac{1}{2}
$$
\normalsize
i.e., for $n\geq 5$, $\~m_\opt(n)$ equals $\lceil 3b(n) - \frac{3}{2} \rceil$ or $\lceil 3b(n) - \frac{3}{2} \rceil + 1$.
\end{main}

This leads us to the following conjecture about our original problem.

\begin{conjmain}[Conjectures \ref{m4-conj1} and \ref{m4-conj2}] \label{mopt-asymp-conj-main}
Let $\mathcal M= \mathcal M_4$.  For $n \ge 5$, we have
\large
\[ m_\opt(n) \ge 3b(n) - \tfrac{3}{2}, \]
\normalsize
and, asymptotically,
\large
\[ m_\opt(n) \sim 3b(n) \sim 3\log_2(n). \]
\normalsize
\end{conjmain}

We briefly touch on the amount of cost savings in this optimization problem in Remark \ref{savings-rem}.

Finally, in Section \ref{nonunif-sec}, we make some comments about the problem for non-uniform
distributions.  In particular, we expect that, as one varies the underlying distribution,
 $m_\opt(n)$ is maximized for the uniform distribution.

\begin{remark} \label{model-comparison-rem}
Based on the above factors, one would expect that the optimal cleanup point should be
greater in the case of complete memory versus no memory, as well as in the case of 
unnumbered shelves versus numbered shelves.  Consequently, we expect that
\[ m_\opt(n; \mathcal M_1) \le m_\opt(n; \mathcal M_2) \le m_\opt(n; \mathcal M_4) \le m_\opt(n; \mathcal M_3). \]
We verified this numerically for small $n$, but we do not focus on this here.  
In particular, we note that preliminary numerics for $\mathcal M_3$ suggest 
 $m_\opt(n) \sim 4 \log_2(n)$ (Remark \ref{M3-calc-rem}).
(In this paper, by ``numerical calculations'' we mean that we used
 high-precision floating point calculations in PARI/GP, and not to mean that
our calculations were provably correct.)
\end{remark}

%
%
%
%

%
%
%
%

\section{Expectation costs} \label{sec2}

%
%
%
%

 In this paper, $m$ and $n$ denote integers satisfying $1 \le m \le n$.  Further, unless
 otherwise specified, $\chi$ will denote a path in $\mathcal X_m$.  If $f$ is a function on 
 $\mathcal X_m$, we sometimes denote $E[f]$ by $E_m[f]$ to specify $m$, or 
 $E_{m,n}[f]$ if we want to specify both $m$ and $n$.
 
 In this section, we decompose 
\[ F(m) = F_{\mathcal L}(m)+F_{\mathcal P}(m) + F_C(m), \]
where the terms on the right will represent average list search, average pile search and average cleanup
costs.  We will analyze these terms individually.  (In the case of no memory, where one does
a list search then a pile search for an object $X \in \mathcal P$, we include both of these
search costs in the function $F_{\mathcal P}$.)
 It appears that $F_{\mathcal L}$ and $F_C$ are
increasing in $m$, whereas $F_{\mathcal P}$ is decreasing in $m$ 
(cf.\ Remark \ref{Eml-rem} and Lemma \ref{FLdec-lem}).  We also expect that $F$ is unimodal---initially decreasing, then
increasing.  Thus our optimization problem is about
the question of when $F_{\mathcal P}$ begins increasing faster than $F_{\mathcal L}+F_C$ decreases.

%

\subsection{Expected search cost}\hfill \label{esc-sec}

%

In this section, we want to find a way to calculate
$E \left[ \f{S}{\ell} \right]$.
We can reduce this to studying
averages of the form 
\begin{equation}
 \sum_{\chi \in \mathcal X_m^{(\ell)}} S(\chi)  =  
\sum_{\chi \in \mathcal X_m^{(\ell)}}
 \sum_{j=1}^{\ell} S(\chi_j; \mathcal P_\chi(j-1) ),
\end{equation}
 where 
\begin{equation}
 \mathcal X_m^{(\ell)} = \{ \chi \in \mathcal X_m : \ell(\chi) = \ell \}.
\end{equation}
Namely, note the probability that $\chi \in \mathcal X_m^{(\ell)}$ depends only on $\ell$, and is
\begin{equation}
 \mu (\mathcal X_m^{(\ell)} )= \f {|\mathcal X_m^{(\ell)}|}{n^{\ell}}.
\end{equation}
Hence
\begin{equation}
F_S(m):= E \left[ \f{S}{\ell} \right]  = \sum_{\chi \in \mathcal X_m} \f{S(\chi)}{\ell(\chi)} \mu(\chi) = 
\sum_{\ell=m}^\infty \f 1{\ell n^{\ell}} \sum_{\chi \in \mathcal X_m^{(\ell)}} S(\chi)
\end{equation}

\begin{prop} \label{prop-21} We have
\[ |\mathcal X_m^{(\ell)} | = m! \leg{n}{m} \ssk{\ell - 1}{m-1}. \]
\end{prop}

Here $\ssk{\ell}{m}$ denotes the Stirling number of the second kind, i.e., the number of ways to partition
a set of $\ell$ elements into $m$ nonempty subsets.

\begin{proof} Note $\mathcal X_m^{(\ell)}$ is in bijection with the set of pairs $(\alpha, X)$ where
$\alpha$ is a sequence of length $\ell -1$ in $\bf X$ consisting of $m-1$ distinct elements, and $X$ is an element of
$\bf X$ not occurring in $\alpha$.  Say the elements occurring in $\alpha$ are $X_{i_1}, \ldots, X_{i_{m-1}}$.
Fix $X_{i_1}, \ldots, X_{i_{m-1}}$ and restrict to such $\alpha$ where the first occurrence of $X_{i_j}$ is
before the first occurrence of $X_{i_k}$ if $j < k$.   Then the number of such $\alpha$ is precisely
$\ssk{\ell - 1}{m-1}$ as each such $\alpha$ can be associated uniquely with a partition of $\{ 1, \ldots, \ell-1 \}$ 
into $m-1$ nonempty subsets---namely the subset associated to $X_{i_j}$ is simply the positions of $\alpha$
at which it occurs.

Hence the total number of $\alpha$ associated to elements in $X_m^{(\ell)}$ is $\ssk{\ell - 1}{m-1}$
times the number of possible orderings for $X_{i_1}, \ldots, X_{i_{m-1}}$ times the total number of choices
for these $m-1$ objects, i.e.,
\[ (m-1)! \leg{n}{m-1} \ssk{\ell -1}{m-1}. \]
Lastly, for each such $\alpha$, there are $n-(m-1)$ distinct choices for $X$.
\end{proof}

As an aside, we note this provides a proof of the identity 
(cf. \cite[Thm 2.11]{charalam})
\begin{equation}
 \sum_{\ell=m}^\infty  \ssk{\ell - 1}{m-1} \frac 1 {n^{\ell}} = \f{(n-m)!}{n!} 
\end{equation}
since $\sum |\mathcal X_m^{(\ell)}|/n^{\ell} = 1$.

\medskip
Write
\[ S(\chi) = S_{\mathcal L}(\chi) + S_{\mathcal P}(\chi) \]
where
\[ S_{\mathcal L}(\chi) = \sum_{\chi_j \not \in \mathcal P_{\chi}(j-1)} S(\chi_j; \mathcal P_\chi(j-1)) \]
and
\[ S_{\mathcal P}(\chi) = \sum_{\chi_j \in \mathcal P_{\chi}(j-1)} S(\chi_j; \mathcal P_\chi(j-1)). \]
In other words, $S_{\mathcal L}(\chi)$ (resp.\ $S_{\mathcal P}(\chi)$) is the total cost of searches 
along $\chi$ when the sought-after object is in $\mathcal L$ (resp.\ $\mathcal P$).

\medskip
The action of the symmetric group $Sym(\mathbf X)$ on $\bf X$ induces an action on $\mathcal X_m^{(\ell)}$.
Namely, for $\sigma \in Sym(\mathbf X)$,  put
\[ \chi^\sigma = ( \chi_1^\sigma, \ldots, \chi_\ell^\sigma ). \]
Also, for $\chi \in \mathcal X_m$, we put $\tau_\chi(j)$ to be the number of times one searches along $\chi$ for an object in $\mathcal P$ when $\mathcal P$ has size $j$. Explicitly, set
$t_0=t_0(\chi)=0$ and, for $1\le j \le m$, let  $t_j=t_j(\chi)$ 
be the minimal integer such that $|\mathcal P_\chi(t_j)|=j$.  Then for $0 < j < m$,
we set $\tau_j(\chi) = t_{j+1}(\chi) - t_j(\chi)-1$.  

\begin{lemma}
For any $\chi \in \mathcal X_m^{(\ell)}$, we have the following average cost formulas:
\[ \f 1{n!} \sum_\sigma S_{\mathcal L}(\chi^\sigma) = \sum_{j=0}^{m-1} s_{\mathcal L}(n-j) \]
and
\[ \f 1{n!} \sum_\sigma S_{\mathcal P}(\chi^\sigma) = \sum_{j=0}^{m-1} \tau_j(\chi) s_{\mathcal P}(j). \] 
\end{lemma}

\begin{proof}
To see the first equality, observe that for any $\chi$, there must be exactly one search for an object $X_{i_j}$ 
which is in $\mathcal L$ when  $\mathcal L$ has $n-j$ objects for each $j=0, 1, \ldots, m-1$.
Fixing one such $j$ and averaging the contribution of this search cost over the permutations $\sigma$
yields $s_{\mathcal L}(n-j)$.

The second equality is similar.
\end{proof}

This yields the following expected cost formulas:
\[ E[ S_{\mathcal L}(\chi) \, | \, \chi \in \mathcal X_m^{(\ell)} ]
= |\mathcal X_m^{(\ell)}|^{-1} \sum_{\chi \in \mathcal X_m^{(\ell)}} S_{\mathcal L}(\chi) = 
\sum_{j=0}^{m-1} s_{\mathcal L}(n-j) \]
and
\[ E[ S_{\mathcal P}(\chi) \, | \, \chi \in \mathcal X_m^{(\ell)} ]
= |\mathcal X_m^{(\ell)}|^{-1} \sum_{\chi \in \mathcal X_m^{(\ell)}} S_{\mathcal P}(\chi) = 
\sum_{j=0}^{m-1} E[\tau_j \, | \,  \mathcal X_m^{(\ell)} ] s_{\mathcal P}(j). \]
Consequently, one has

\begin{lemma}
\[ F_S(m) = E \left[ \f{S}{\ell} \right] = F_{\mathcal L}(m) + F_{\mathcal P}(m) \]
where
\[ F_{\mathcal L}(m) = 
 \sum_{\ell=m}^\infty \f {|\mathcal X_m^{(\ell)}|}{\ell n^\ell} 
 \sum_{j=0}^{m-1} s_{\mathcal L}(n-j) 
= E_m \left[ \f 1{\ell} \right] \sum_{j=0}^{m-1} s_{\mathcal L}(n-j) \]
and
\[ F_{\mathcal P}(m) =
\sum_{\ell=m}^\infty \f {|\mathcal X_m^{(\ell)}|}{\ell n^\ell} 
 \sum_{j=1}^{m-1} E_m[\tau_j \, | \, \mathcal X_m^{(\ell)} ] s_{\mathcal P}(j). \]
\end{lemma}
 
%

\subsection{Expected cleanup cost}\hfill \label{ecc-sec}

%

The expected cleanup cost per item is simply
\[ 
F_C(m) := E_m \left[ \f{C}{\ell} \right] = \sum_{\chi \in \mathcal X_m} \f {C(\chi)}{\ell(\chi)} \mu(\chi)
= C_m \sum_{\chi \in \mathcal X_m} \f {\mu(\chi)}{\ell(\chi)}  = C_m E_m \left[ \f 1{\ell} \right] \]
where
$C_m$ is as in \eqref{Cmnum} or \eqref{Cmunnum} according to whether the shelves are numbered or not.

With this notation, the expected search-and-cleanup cost is
\[ F(m) = F_{\mathcal L}(m) + F_{\mathcal P}(m) + F_C(m). \]
Note that in the case of numbered shelves, we have
\[ F_C(m) = F_{\mathcal L}(m). \]

In the case of unnumbered shelves,
\[ F_C(m) = E_m \left[ \f 1{\ell} \right] \sum_{j=0}^{m-1} b_f(n-j-1) \]
and
\[ F_{\mathcal L}(m) = E_m \left[ \f 1{\ell} \right] \sum_{j=0}^{m-1} b(n-j). \]
Consequently, we have
\[ F_C(m) = \left( 1- \frac{m-\sum_{j=0}^{m-1} \log_2(n-j+1)/(n-j)}
{\sum_{j=0}^{m-1} \log_2(n-j+1) } \right) F_{\mathcal L}(m). \]
We remark that this implies $F_C(m) \le F_{\mathcal L}(m).$

In any case, we have  reduced our problem to studying the expected list search cost
$F_{\mathcal L}(m)$ and $F_{\mathcal P}(m)$.

\subsection{Some simple calculations} \label{simp-calc-sec} \hfill

Here, we calculate $F(1;n)$ and $F(2;n)$ for each of the four models discussed above,
which we hope will be instructive.
In all cases, these calculations, together with the observation \eqref{eq:E2ell} that
$E_{2}\left[ \frac 1 \ell \right] < \frac 12$, imply that $F(2;n) < F(1;n)$ for all $n \ge 2$, giving
one proof of Proposition \ref{propmain1}.
We remark the proof of this inequality does not depend on the specific definitions of the functions 
$b(j)$ and $b_f(j)$, just that these functions are increasing.

\subsection*{Calculations for $m=1$} \hfill

First consider $m=1$.  Then $\mathcal X_1 = \mathcal X_1^{(1)} =  \{ (1), (2), (3), \ldots, (n) \}$.
Consequently $E[\frac 1 \ell]=1$ and $F_{\mathcal P}(1) = 0$.

\subsubsection{Unnumbered shelves}

Suppose we have unnumbered shelves, i.e., $\mathcal M_1$ or $\mathcal M_3$.  Then
\[ F(1;n) = F_{\mathcal L}(1;n) + F_C(1;n) = b(n)+b_f(n-1). \]

\subsubsection{Numbered shelves}

Suppose we are in the case of numbered shelves, i.e., $\mathcal M_2$ or $\mathcal M_4$.
Then
\[ F(1;n) = 2F_{\mathcal L}(1;n) = 2s_{\mathcal L}(n) = 2b(n). \]

\subsection*{Calculations for $m=2$}\hfill

Now take $m=2$.  Then, for $\ell \ge 2$, $\mathcal X_2^{(\ell)}$ consists of the 
$2 \bmx n \\ 2 \emx$ sequences of the form $(X_1, X_1, \ldots, X_1, X_2)$, where there are
$\ell-1$ occurrences of $X_1$.  Then
\begin{equation} \label{eq:E2ell}
 E\left[ \frac 1 \ell \right] = n(n-1) \sum_{\ell=2}^\infty \frac 1{\ell n^\ell}
< n(n-1) \frac 12 \sum_{\ell=2}^\infty \frac 1{n^\ell} = \frac 12. 
\end{equation}
As an aside, we note that the equality in \eqref{eq:E2ell} yields the closed form expression
\begin{equation}
 E_2\left[ \frac 1 \ell \right] = n(n-1) \left( \log \frac 1{1-1/n} - \frac 1 n \right).
\end{equation}

Since
\[ E[\tau_\chi(1) \, | \, \chi \in \mathcal X_2^{(\ell)} ] = \ell - 2, \]
we have
\[ F_{\mathcal P}(2;n) = s_{\mathcal P}(1) \sum_{\ell=m}^\infty \frac{|\mathcal X_2^{(\ell)}|}{\ell n^\ell} (\ell - 2)
= s_{\mathcal P}(1) \left( 1 - 2  E\left[ \frac 1 \ell \right] \right). \]

\subsubsection{No memory, unnumbered shelves}

Suppose $\mathcal M= \mathcal M_1$, so
\[ F(2; n) = F_{\mathcal L}(2; n) + F_{\mathcal P}(2;n) + F_C(2;n).  \]
Here 
$F_{\mathcal L}(2;n) = E\left[ \frac 1 \ell \right]  (b(n)+b(n-1))$,
$F_C(2;n) = E\left[ \frac 1 \ell \right] (b_f(n-1) + b_f(n-2))$, and
$s_{\mathcal P}(1) = b_f(n-1)+1$, so
\[ F(2; n) = b_f(n-1)+1 +E\left[ \frac 1 \ell \right]
\left( b(n)+b(n-1)-b_f(n-1)+b_f(n-2)-2 \right). \]

\subsubsection{No memory, numbered shelves}

Suppose $\mathcal M= \mathcal M_2$, so
\[ F(2; n) = 2F_{\mathcal L}(2; n) + F_{\mathcal P}(2;n).   \]
Here 
$F_{\mathcal L}(2;n) = 2 E\left[ \frac 1 \ell \right] b(n)$ and
$s_{\mathcal P}(1) = b_f(n)+1$, so
\[ F(2; n) = b_f(n)+1 + E\left[ \frac 1 \ell \right]
\left( 4b(n)-2b_f(n)-2  \right). \]

\subsubsection{Complete memory, unnumbered shelves}

Suppose $\mathcal M= \mathcal M_3$, so
\[ F(2; n) = F_{\mathcal L}(2; n) + F_{\mathcal P}(2;n) + F_C(2;n).  \]
Here
$F_{\mathcal L}(2;n) = E\left[ \frac 1 \ell \right]  (b(n)+b(n-1))$,
$F_C(2;n) = E\left[ \frac 1 \ell \right] (b_f(n-1) + b_f(n-2))$, and
 $s_{\mathcal P}(1) = 1$, so
\[ F(2; n) = 1 + E\left[ \frac 1 \ell \right]
\left( b(n)+b(n-1)+b_f(n-1)+b_f(n-2)-2 \right). \]

\subsubsection{Complete memory, numbered shelves}

Suppose $\mathcal M= \mathcal M_4$, so
\[ F(2; n) = 2F_{\mathcal L}(2; n) + F_{\mathcal P}(2;n).  \]
Here
$F_{\mathcal L}(2;n) = 2 E\left[ \frac 1 \ell \right] b(n)$ and
$s_{\mathcal P}(1) = 1$, so
\[ F(2; n) = 1 + E\left[ \frac 1 \ell \right]
\left( 4b(n) - 2 \right). \]
 
%

\subsection{Expected list search cost}\hfill

%

We now return to studying search costs, in particular we consider the expected list search and
cleanup costs, $F_{\mathcal L}(m)$ and $F_C(m)$.
Since
\[ F_{\mathcal L}(m) = E_m \left[ \f 1{\ell} \right] \sum_{j=0}^{m-1} s_{\mathcal L}(n-j) \quad \text{and}
\quad F_C(m) = C_m E_m \left[ \f 1{\ell} \right], \]
studying these quantities reduces to studying 
\begin{equation}
 E_m \left[ \f 1{\ell} \right] = E_{m,n} \left[ \f 1{\ell} \right] = \sum_{\ell =m}^\infty \frac{|\mathcal X_m^{(\ell)}|}{\ell n^\ell}
= m! \bmx n \\ m \emx \sum_{\ell =m}^\infty \ssk{\ell-1}{m-1} \frac 1 {\ell n^\ell}.
\end{equation}
Note this is the expected value of the reciprocal of a waiting time for a sequential occupancy problem. 

First we obtain the following finite formula, which allows us to compute 
$E_m \left[ \f 1{\ell} \right]$ quickly.
 
\begin{lemma} We have
\begin{equation} \label{Eml-correct}
E_m \left[\frac{1}{\ell}\right] = \frac{n!}{(n-m)!} (-1)^{m+1} + m \bmx n\\ m \emx \sum_{j=1}^{m-1} (-1)^{m-j} \bmx m-1\\ j \emx \frac{\log(1 - j/n)}{j}.
\end{equation}
\end{lemma}

Note the first term can be interpreted as the $j=0$ term for the sum on the right.

\begin{proof}
The generating function for Stirling numbers of the second kind is given by 
\begin{equation}
\sum_{\ell \geq m}^\infty \ssk{\ell - 1}{m-1} x^{\ell-1} = \frac{x^{m-1}}{(1-x)(1-2x)\cdots (1-(m-1)x)}.
\end{equation}
Hence
\begin{equation} \label{Eml-intrep}
E_m\left[ \frac 1 \ell \right] = \frac{n!}{(n-m)!}
 \int_0^{1/n} \frac{x^{m-1}}{(1-x)(1-2x)\cdots (1-(m-1)x)} \, dx.
\end{equation}
We compute the integral using partial fractions.
$$
\frac{x^{m-1}}{(1-x)(1-2x)\cdots (1 - (m-1)x)} = \frac{A_1 x}{1-x} + \frac{A_2 x}{1 - 2x} + \cdots + \frac{A_{m-1} x}{1 - (m-1)x},
$$
where $\displaystyle{A_j = \frac{(-1)^{m-1-j}}{(j-1)! \,(m-1-j)!}}$. Now $\displaystyle{\int \frac{A_j x}{1 - jx}\, dx = -\frac{A_j x}{j} - \frac{A_j\, \log(1 - jx)}{j^2}}$, and so we have
\begin{equation}
E_m \left(\frac{1}{\ell}\right) = m!\, \bmx n\\ m \emx \sum_{j=1}^{m-1} -\frac{A_j}{j}\cdot \frac{1}{n} - \frac{A_j}{j^2} \cdot \log(1 - j/n)
\end{equation}
We can simplify this sum a little by observing that $\displaystyle{m!\, \frac{A_j}{j} = m\cdot (-1)^{m-1-j} \bmx m-1\\j \emx}$. Then the first part of the above summation simplifies to
$$
\frac{m}{n} \bmx n\\m \emx (-1)^m \sum_{j=1}^{m-1} (-1)^j \bmx m-1\\j \emx = \frac{m}{n} \bmx n\\ m \emx (-1)^m\cdot (-1)
$$
Putting all this together yields the lemma.
\end{proof} 

Since the above formula is an alternating sum, it not so useful in studying the behavior of $E_{m} \left[ \frac 1\ell \right]$ as $m$ varies, which is our goal, though it is useful for numerics.

Now we observe some elementary bounds.

\begin{lemma}  \label{jensen}
For $1 \le m \le n$,
\[ \frac 1{n( \log(n) - \log(n-m) )} \le \frac 1{E_m[\ell]} \le E_m \left[ \f 1{\ell} \right] \le 
 \frac 1m.  \]
\end{lemma}

(We interpret the leftmost term as $0$ when $m=n$.)

\begin{proof}
As is well known, $\ell$ is a sum of $m$ independent geometric distributions with means
$\frac nn$, $\frac n{n-1}$, \ldots, $\frac n{n-m+1}$.  Thus
 \begin{equation}
  E_m[\ell] = \sum_{j=0}^{m-1} \f n{n-j} = n (H_n - H_{n-m}),
 \end{equation}
 where $H_j$ is the $j$-th harmonic number.  This implies the first inequality.  The second is Jensen's inequality.  The third follows as $\ell(\chi) \ge m$ for any $\chi \in \mathcal X_m$.
\end{proof}

If $n \to \infty$ and $m$ grows slower than $\sqrt n$, $\ell \to E_m[\ell]$ in probability \cite{BB}.
Thus we might expect $\frac 1{E_m[\ell]}$ to be a good approximation for $E_m[ \frac 1 \ell]$, as the following
bound shows.

\begin{lemma} \label{cheby-bound}
 For $1 < m < \sqrt{2n}$,
\[ E_m \left[ \frac 1 \ell \right] < \frac 1{E_{m-1}[\ell]}. \]
For $1 < m \le n$, 
\[ E_m \left[ \frac 1 \ell \right] \le \frac 1{E_{m-1}[\ell]} + \frac{(m-1)(n-m)}{2n^2}. \]
\end{lemma}

\begin{proof}
Let $x > 0$.
Chebyshev's inequality tells us
\[ Pr(|\ell-E_m[\ell]| \ge x) \le \frac{\textit{Var}(\ell)}{x^2}. \]
Consequently,
\begin{equation} \label{cheby2}
 E_m \left[ \frac 1 \ell \right] \le \frac 1{E_m[\ell]-x} \left( 1 - y \right)
+ \frac ym = \frac 1{E_m[\ell] - x} + y \left( \frac 1 m - \frac 1{E_m[\ell] - x}  \right)
\end{equation}
for any $y \ge \frac{\textit{Var}(\ell)}{x^2}$.
Note 
\[ \textit{Var}(\ell) = \sum_{j=0}^{m-1} \frac{jn}{(n-j)^2} \le \frac{n}{(n-m+1)^2} \frac{(m-1)m}2. \]
Set $x= \frac{n}{n-m+1}$ so $E_m[\ell] - x = E_{m-1}[\ell]$.
 Now we apply \eqref{cheby2} with
\[ y = \frac{(n-m+1)^2}{n^2}  \frac{n}{(n-m+1)^2} \frac{(m-1)m}2 = \frac{(m-1)m}{2n}. \]
Observe
\[ \frac 1 m - \frac 1{E_{m-1}[\ell]} =\frac 1{mE_{m-1}[\ell]} \left( \sum_{j=1}^{m-2} \frac j{n-j} - 1 \right)
\le \frac 1{mE_{m-1}[\ell]} \left( \frac{(m-2)(m-1)}{2(n-m+2)} -1 \right), \]
which is negative if $m < \sqrt{2n}$, and one gets the first part.  The second part
follows from the crude bound
\[ \frac 1{mE_{m-1}[\ell]} \left( \sum_{j=1}^{m-2} \frac j{n-j} - 1 \right) \le \frac 1m \frac {n-m}n. \]
\vspace{-1cm}
\end{proof}

The above two lemmas imply the sequence $E_m \left[ \frac 1\ell \right]$
is decreasing in $m$ when $m < \sqrt{2n}$.

\begin{remark} \label{Eml-rem}
 We in fact expect the first part of Lemma \ref{cheby-bound} to hold for all $1 < m \le n$, as well
as the stronger bound $E_m \left[ \frac 1 \ell \right] \le \frac {m-1}m \frac 1{E_{m-1}[\ell]}$, which appears
true numerically.  This stronger bound together with Lemma \ref{jensen} would imply
\begin{equation} \label{mEm-dec}
 mE_{m} \left[ \frac 1\ell \right] > (m+1) E_{m+1} \left[ \frac 1\ell \right],
\end{equation}
which means that for any of our four models $\mathcal M$,
the expected list search and cleanup costs, $F_{\mathcal L}(m)$ and $F_C(m)$,
 are strictly decreasing in $m$.  
Furthermore, numerics 
suggest that the sequence
$\frac 1{E_m[\ell]}$ decreases in $m$ faster than the sequence 
$E_m[\frac 1 \ell]$.   For instance, it appears
\[  mE_{m} \left[ \frac 1 \ell \right] - (m+1)E_{m+1} \left[ \frac 1 \ell \right] \le
\frac{m}{E_{m}[\ell]} - \frac {m+1}{E_{m+1}[\ell]}; \quad \text{and} \quad
 \frac{E_{m+1}[\ell]}{E_{m}[\ell]} \ge \frac{E_{m}[1/\ell]}{E_{m+1}[1/\ell]}. \]
\end{remark}

\begin{lemma} \label{FLdec-lem}
Fix $m \ge 1$, and let $\epsilon > 0$.  Then for sufficiently large $n$,
\[ F_{\mathcal L}(m;n) > F_{\mathcal L}(m+1;n) - \epsilon \quad \text{and} \quad
F_C(m;n) > F_C(m+1; n) - \epsilon. \]
In the case of unnumbered shelves, we may take $\epsilon = 0$, i.e., for $m_0$ sufficiently
small relative to $n$, $F_{\mathcal L}(m; n)$ and $F_C(m;n)$ are decreasing in $m$ for $1 \le m < m_0$.
\end{lemma}

\begin{proof} 
It is easy to see that (e.g., by Lemma \ref{jensen} or \cite{BB}), for fixed $m$,
$\frac 1{E_{m,n}[\ell]}$ and
$E_{m,n} \left[ \f 1{\ell} \right]$ are increasing sequences with limit $\frac 1m$.
This implies that for $n$ sufficiently large
\[  E_{m,n} \left[ \f 1{\ell} \right]  > \frac{m+1}m E_{m+1,n} \left[ \f 1{\ell} \right] - \epsilon.  \]
This implies the lemma as
\[  \frac{m+1}m \sum_{j=0}^{m-1} s_{\mathcal L}(n-j) - \sum_{j=0}^{m} s_{\mathcal L}(n) \ge 0 \quad \text{and} \quad  \frac{m+1}m C_m - C_{m+1} \ge 0, \]
and these differences can be bounded away from 0 in the unnumbered case.
\end{proof}

%

\subsection{Expected pile search cost}\hfill

%

Now we consider the expected pile search cost

\[ F_{\mathcal P}(m) = \sum_{j=1}^{m-1} E_m \left[\frac{\tau_j}{\ell}\right] s_{\mathcal P}(j) = 
\sum_{\ell=m}^\infty \f {|\mathcal X_m^{(\ell)}|}{\ell n^\ell} 
 \sum_{j=1}^{m-1} E_m[\tau_j(\chi) \, | \, \chi \in \mathcal X_m^{(\ell)} ] s_{\mathcal P}(j). \]

First we remark the following explicit formula for the expected values in the inner sum.

 \begin{lemma} For $1 \le j \le m-1$, we have
 \[ E_m[\tau_j \, | \,  \mathcal X_m^{(\ell)} ]= \ssk{\ell-1}{m-1}^{-1} \times
 \sum_{k=1}^{\ell-m}  j^k  \ssk{\ell-k-1}{m-1}. \] 
 \end{lemma} 

\begin{proof}
If $\ell = m$, then $\tau_j(\chi) = 0$ for all $\chi \in \mathcal X_m^{(\ell)}$ and the formula trivially holds,
so suppose $\ell > m$ and let $\chi \in \mathcal X_m^{(\ell)}$.  If  $r = \tau_j(\chi) \ge 1$,
we can remove the element at position $t_{j+1}-1$ to get an element $\chi' \in \mathcal X_m^{(\ell-1)}$
such that
\[ \tau_{j}(\chi') = r-1, \, \tau_{j'}(\chi') = \tau_{j'}(\chi) \quad j' \ne j. \]
This map from $\chi$ to $\chi'$ is a $j$-to-1 surjective map, i.e., for $j, r \ge 1$ we have
\[ \# \{ \chi \in \mathcal X_m^{(\ell)} : \tau_j(\chi) = r \} = j \cdot 
\# \{ \chi \in \mathcal X_m^{(\ell-1)} : \tau_j(\chi) = r-1 \}. \] 
Summing over all $r \ge 1$, we see
\[ \# \{ \chi \in \mathcal X_m^{(\ell)} : \tau_j(\chi) \ge 1 \} = j |\mathcal X_m^{(\ell-1)}|. \]
Similarly if $r \ge k$ we can remove the last $k$ elements before position $t_{j+1}$ to get a
$j^k$-to-1 map into $\mathcal X_m^{(\ell-k)}$, and we have
\begin{equation} \label{Emtau-cdf}
  \# \{ \chi \in \mathcal X_m^{(\ell)} : \tau_j(\chi) \ge k \} = j^k |\mathcal X_m^{(\ell-k)}|.
\end{equation}
Now observe that
\[ E_m[\tau_j \, | \,  \mathcal X_m^{(\ell)} ] = 
\sum_{k=1}^{\ell-m} k \mu \{ \chi \in \mathcal X_m^{(\ell)} : \tau_j(\chi) = k \} =
\sum_{k=1}^{\ell-m}  \mu \{ \chi \in \mathcal X_m^{(\ell)} : \tau_j(\chi) \ge k \} \]
and apply Proposition \ref{prop-21}.
\end{proof}

Consequently, we have 
 \begin{align} \nonumber
 E_m  \left[ \frac {\tau_j}{\ell} \right] 
 &= m! \bmx n \\ m \emx \sum_{\ell=m}^\infty \frac 1{\ell n^\ell} \sum_{k=1}^\infty j^k \ssk{\ell-k-1}{m-1} \\ \nonumber
 &= m! \bmx n \\ m \emx \sum_{k=1}^\infty j^k \sum_{\ell = m+k}^\infty \frac 1{\ell n^\ell} \ssk{\ell-k-1}{m-1} \\ \label{Etauj-expr2}
  &= m! \bmx n \\ m \emx \sum_{k=1}^\infty \frac{j^k}{n^k} \sum_{\ell = m}^\infty \frac 1{(\ell+k) n^{\ell}} \ssk{\ell-1}{m-1}
\end{align}
  This expression allows us to get the following upper bound.

\begin{prop} \label{prop:cov}
For $1 \le j \le m-1$, the covariance $\mathrm{Cov}(\tau_j, \frac 1{\ell}) < 0$, i.e., 
\[ E_m \left[ \frac {\tau_j}{\ell} \right] < \frac{j}{n-j} E_m \left[ \frac 1 \ell \right]. \]
\end{prop}

\begin{proof}
From \eqref{Etauj-expr2} we have
\[ E_m  \left[ \frac {\tau_j}{\ell} \right]  < \sum_{k=1}^\infty \frac {j^k}{n^k} E_m \left[ \frac 1 \ell \right] =  \frac{j}{n-j} E_m \left[ \frac 1 \ell \right]. \]
That this is equivalent to the condition of negative covariance asserted above follows as
\[ \mu \{ \chi  \in \mathcal X_m : \tau_j(\chi) = k \} = \left( \frac jn \right)^k \frac{n-j}n \]
implies
\[ E_m[\tau_j] =  \frac{n-j}n \sum_{k = 1}^\infty k \left( \frac jn \right)^k  
= \frac j n \frac{n-j} n \left(1- \frac jn \right)^{-2} = \frac j{n-j}. \]
\end{proof}

\begin{remark} Suppose $n \ge 4$ and $1 \le j < m < n$.  Then numerically it appears that
\[ E_m \left[ \frac {\tau_j}{\ell} \right]  \le E_{m+1} \left[ \frac{\tau_{j+1}}{\ell} \right]. \]
This would imply that $F_{\mathcal P}(m)$ is increasing in $m$, and, in the case of complete memory,
\[ F_{\mathcal P}(m+1) > \frac{m+1}{m} F_{\mathcal P}(m). \]
\end{remark}

Lastly we note

\begin{lemma}  \label{prob1-lem} We have
\[ mE_m \left[ \frac 1 \ell \right] + \sum_{j-1}^{m-1} E_m \left[ \frac{\tau_j}{\ell} \right] = 1. \]
\end{lemma}

\begin{proof} This follows from the observation that $\sum_{j-1}^{m-1}\tau_j(\chi) = \ell(\chi) - m.$
\end{proof}

\subsection{Reinterpreting $F(m)$} \hfill

From above, we can rewrite
\begin{equation}
F(m) =  m E_m \left[ \frac 1 \ell \right] s^*(m)  + \sum_{j=1}^{m-1} E_m \left[ \frac{\tau_j} \ell \right] s_{\mathcal P}(j),
\end{equation}
where
\[ s^*(m) = \frac{\sum_{j=0}^{m-1} s_{\mathcal L}(n-j) + C_m}m \]
denotes the average total (search plus cleanup) cost of taking an object out of $\mathcal L$.
This expression yields the following interpretation (cf.\ Lemma \ref{prob1-lem}):
 we can think of $m E_m\left[ \frac 1 \ell \right]$
as the probability that a given search will cost $s^*(m)$, and $E_m \left[ \frac{\tau_j} \ell \right]$
as the probability that a given search will cost $s_{\mathcal P}(j)$.

It is easy to see that $m E_m\left[ \frac 1 \ell \right] = 1$ if and only if $m=1$, so we have
$F(1) > F(m)$ whenever $m > 1$ satisfies $s_{\mathcal P}(m-1) < s^*(m)$.  This yields

\begin{lemma} \label{F1comp-lem}
Let $1 < m \le n$.  
\begin{enumerate}
\item In the case of unnumbered shelves, if $m < 4b(n-m)$, then $F(m) < F(1)$.

\item In the case of numbered shelves, if $m < 4b(n)$, then $F(m) < F(1)$.
\end{enumerate}
\end{lemma}

\section{An upper bound} \label{M4-sec}

For simplicity now, we will assume we are in model $\mathcal M_4$ (complete memory, numbered shelves), though a similar argument can be used for $\mathcal M_2$ as well.
In this case we have
\[ F(m) = 2F_{\mathcal L}(m) + F_{\mathcal P}(m). \]

Recall that $F(1) = 2b(n)$.  Thus, once the pile $\mathcal P$ has more than $4b(n)$ elements,
a single average pile search must cost more than $F(m_\opt(n))$.  This idea gives the following
upper bound.

\begin{prop} \label{ub-prop} Suppose $n \ge 1$.  Then $m_\opt(n) < 4b(n)$.
\end{prop}

We will first prove a lemma.  
We say a function $f: \mathcal X_m \to \mathbb R$ is {\em additive} if, for any 
$\chi = (\chi_1, \ldots, \chi_\ell)$ and any $1 \le k < \ell$, we can write $f$ as a sum of terms
\[ f(\chi) = \sum_{j=1}^\ell f(\chi_j; \mathcal P_\chi(j-1)) \]
where the $f(\chi_j; \mathcal P_\chi(j-1))$ depends only upon $\chi_j$ and what is in the pile
before time $j$.  We can naturally restrict such functions $f$ to functions of $\mathcal X_{k}$
for $k < m$.
Note that all the cost functions we considered above are additive, and any linear combination
of additive functions is additive.

Let $m \ge k \ge 1$. 
 Define a {\em restriction map}
 $R^m_{k}: \mathcal X_{m} \to \mathcal X_{k}$, given by
\[ R^m_k \chi = (\chi_1, \ldots, \chi_{t_{k}(\chi)}). \]
We let $T^m_{k}: \mathcal X_m \to \bigcup_{k=1}^\infty \mathbf X^k$ be the truncated tail
from the restriction map, i.e.,
\[ T^m_{k} \chi = (\chi_{t_{k}(\chi)+1}, \ldots, \chi_{\ell(\chi)}). \]
Put $\mathcal P_{\chi,k} = \mathcal P_\chi(t_{k}(\chi))$
 to be the pile after time $t_k(\chi)$.

\begin{lemma} Suppose $f : \mathcal X_m \to \mathbb R$ is additive and $m > k \ge 1$.  If 
\begin{equation} \label{Efl-ineqcond}
f(X_i; \mathcal P_{\chi,k}) \ge \frac{f(R_k^m\chi)}{\ell(R_k^m\chi)}
\end{equation}
 for all $\chi \in \mathcal X_m$ and $X_i \in \mathbf X$, then 
 \begin{equation} \label{Efl-ineq}
 E_m \left[ \frac f\ell \right] \ge E_k \left[ \frac f\ell \right].
\end{equation}
 If, further, the inequality in \eqref{Efl-ineqcond} is strict for
 for some $\chi$ and $X_i$, then the inequality in \eqref{Efl-ineq} is also strict.
\end{lemma}

\begin{proof} Since $f$ and $\ell$ are additive, we can write
\[ E_m \left[ \frac f \ell \right] = \sum_{\chi \in \mathcal X_m} \mu(\chi) 
\frac{ f(R^m_k\chi) + f(T^m_k\chi; \mathcal P_{\chi,k}) }{\ell(R^m_k\chi) + \ell(T^m_k\chi)}. \]
Then the above condition guarantees, for any $\chi$,
\[ \frac{ f(R^m_k\chi) + f(T^m_k\chi; \mathcal P_{\chi,m-1}) }{\ell(R^m_k\chi) + \ell(T^m_k\chi)} \ge 
\frac{f(R^m_k\chi)}{\ell(R^m_k\chi)} \]
as $\frac{a+b}{c+d} \ge \frac{a}{c}$ if and only if $\frac bd \ge \frac ac$.
\end{proof}

\begin{proof}[Proof of Proposition \ref{ub-prop}] Set $k=\lfloor 4b(n) \rfloor$ and let $k < m \le n$.
Let  $\bar S_{\mathcal L}(\chi)$ (resp.\ $\bar S_{\mathcal P}(\chi)$) 
be the average of $S_{\mathcal L}(\chi^\sigma)$ (resp.\ $S_{\mathcal P}(\chi^\sigma)$), where $\sigma$ 
ranges over $Sym(\mathbf X)$.  Let
 $f=2\bar S_{\mathcal L} + \bar S_{\mathcal P}$, so $F(m) = E_m[\frac f \ell]$.

Then note that 
\[ \frac{f(R^m_k\chi)}{\ell(R^m_k\chi)} = \frac{2kb(n) +
 \bar S_{\mathcal P}(R_k^m\chi) }{\ell(R_k^m\chi)}  \le 2b(n), \]
since 
\[ \bar S_{\mathcal P}(R^m_k \chi) \le (\ell(R^m_k \chi)-k) s_{\mathcal P}(k-1)
\le 2b(n) (\ell(R^m_k \chi ) -k) \]
Furthermore, this inequality must be strict for some $\chi$ (in fact, one only gets equality when
$n+1$ is a power of 2 and $t_j(\chi)=j$ for $j < k$).

On the other hand, for any $X_i \in \mathbf X$, 
we have $f(X_i ; \mathcal P_{\chi,k}) \ge 2b(n)$ (with
equality if $X_i \not \in \mathcal P_{\chi,k}$).  Applying the above lemma, we
see $F(m) > F(k)$.
\end{proof}

%
\section{An approximate problem}\label{approx-sec}
%
 
Here we make a conjectural lower bound and asymptotic for $m_\opt(n)$ by comparing our problem with a simpler optimization problem.
We continue, for simplicity, in the case of $\mathcal M_4$, though similar approximate
problems could be considered for $\mathcal M_1$, $\mathcal M_2$ and $\mathcal M_3$ also.

Based on the bounds for $E_m[ \frac 1 \ell]$ and $E_m [\frac{\tau_j}{\ell}]$ above, we consider the
approximate expection cost
\[ \~F(m) = 2 \~ F_{\mathcal L}(m) + \~ F_{\mathcal P}(m), \]
where
\[ \~ F_{\mathcal L}(m) = mb(n) \frac 1{E_m[\ell]} \]
and
\[ \~ F_{\mathcal P}(m) = \sum_{j=1}^{m-1} \frac{j+1}2 \frac j{n-j} \frac 1{E_m[\ell]}. \]

Specifically, Lemma \ref{jensen} and Proposition \ref{prop:cov} imply
$F_{\mathcal L}(m) \ge \~ F_{\mathcal L}(m)$
and 
\[ F_{\mathcal P}(m) \le  \sum_{j=1}^{m-1} \frac{j+1}2 \frac j{n-j} E_m \left[ \frac 1 \ell \right] . \]
We suspect that the approximation 
$\~ F_{\mathcal P}(m)$ is much closer to this upper bound for $F_{\mathcal P}(m)$
than $F_{\mathcal P}(m)$ itself is, and so we should have $F_{\mathcal P}(m) \le \~ F_{\mathcal P}(m)$.  This is supported by numerical evidence. (See Table \ref{mopt-tab} for some numerical calculations.) Moreover, since conjecturally $F_{\mathcal L}(m)$ is decreasing in $m$ 
(and, empircally, faster than $\~F_{\mathcal L}(m)$ is),
while $F_{\mathcal P}(m)$ is increasing in
$m$ (and, empirically, slower than $\~F_{\mathcal P}(m)$), we make the following conjecture.

Let $\~m_\opt(n)$ be the value of $m \in \{ 1, 2, \ldots, n \}$ which minimizes $\~ F(m)$.  We call the problem
of determining $\~m_\opt(n)$ an {\em approximate search with cleanup problem}.

\begin{conj} \label{m4-conj1}
 Let $\mathcal M=\mathcal M_4$.  Then $m_\opt(n) \ge \~m_\opt(n)$.
\end{conj}


\begin{table}[!ht]
\caption{Small values of $m_\opt$ and $\~m_\opt$}
 \label{mopt-tab}
\[  \begin{matrix}
n &\vline& 1 & 2 & 3 & 4 & 5 & 6 & 7 & 8 & 9 & 10 &11&12&13&14&15 & 16&17&18&19&20\\
\hline
\~m_\opt&\vline & 1&2&3&4&5&6&7&8&8&8 & 9&9&9&9&9&10&10&10&10&11 \\
m_\opt &\vline& 1&2&3&4&5&6&7&8&8&8 & 9&9&9&10&10 &10&10&10&11&11
\end{matrix} \]

\[ \begin{matrix}
n &\vline& 21&22&23&24&25 &26&27&28&29&30 &31&32&33&34&35\\
\hline
\~m_\opt&\vline& 11&11&11&11&11 &12&12&12&12&12& 12&12&12&13&13 \\
m_\opt &\vline& 11&11&11&12&12 &12&12&12&12&12 &13&13&13&13&13
\end{matrix} \]
\end{table}

\begin{thm} \label{approx-thm} For any $n \ge 5$, we have
\large
\[ 3b(n) -\tfrac{3}{2}  \le \~m_\opt(n) < 3b(n)+ \tfrac{1}{2}. \]
\normalsize
In other words, for $n \ge 5$, $\~m_\opt(n)$ is either $\lceil 3b(n)- \frac 32 \rceil$ or
$\lceil 3b(n)- \frac 32 \rceil+1$.
\end{thm}

We note that in fact both possibilities of this proposition occur: sometimes $\~m_\opt(n)$ is 
$\lceil 3b(n)- \frac 32 \rceil$ and sometimes it is $\lceil 3b(n)- \frac 32 \rceil+1$, though numerically
there seems to be a tendency for
$\~m_\opt(n)$ to be in the right half of this interval, i.e., most of the time $\~m_\opt(n) > 3b(n) - \frac 12$.

\begin{proof}
Let $1 \le m < n$.  We want to investigate when the difference
\begin{equation} \label{deltaFmtilde}
\begin{aligned}
\~F(m) - \~F(m+1) = \left( 2mb(n) + \frac 12 \sum_{j=1}^{m-1} \frac{j(j+1)}{n-j} \right)
&\left( \frac 1{E_m[\ell]} - \frac 1 {E_{m+1}[\ell]} \right)\\
&- \left( 2b(n) + \frac 12 \frac {m(m+1)}{n-m} \right) \frac 1{E_{m+1}[\ell]}
\end{aligned}
\end{equation}
is positive, i.e., when is $\~F(m)$ is decreasing in $m$?  
The above expression is positive if and only if
\[ 2\~F_{\mathcal L}(m) - 2\~F_{\mathcal L}(m+1) > \~ F_{\mathcal P}(m+1) - \~ F_{\mathcal P}(m). \]
Since ${E_{m+1}[\ell]} - E_{m}[\ell]  = \frac n{n-m}$, this is equivalent to 
\begin{equation} \label{Fmtilde-ineq}
4 b(n)(n-m) \left( \frac{nm}{n-m} - E_m[\ell] \right) > m(m+1)E_m[\ell]
- n \sum_{j=1}^{m-1} \frac{j(j+1)}{n-j}.
\end{equation}

The left hand side of \eqref{Fmtilde-ineq} is
\[ 4nb(n) \sum_{j=0}^{m-1} \frac{m-j}{n-j}, \]
whereas the right hand side of \eqref{Fmtilde-ineq} is
\[   n \sum_{j=0}^{m-1} \frac{m(m+1)-j(j+1)}{n-j}  =  n  \sum_{j=0}^{m-1} \frac{(m-j)(m+j+1)}{n-j}. \]
Hence \eqref{Fmtilde-ineq} is positive if and only if
\begin{equation} \label{Fmtilde-ineq2}
 \sum_{j=0}^{m-1} \frac{m-j}{n-j} \left( 4b(n) - (m+j+1) \right) > 0.
\end{equation}

\begin{lemma} Let $a, b, c \in \mathbb Z$ with $a > 1$ and $c \ge \max \{ a, b \}$.  
The sum
\[ \sum_{j=0}^a \frac{(a-j)(b-j)}{c-j} \]
is negative if $b \le \frac {a-2}3$.
 This sum is positive if if $b > \frac a3$ and $c \ge 5a^2$.
\end{lemma}

\begin{proof}
All nonzero terms of the sum are positive (resp.\ negative) if $b \ge a$ (resp.\ $b \le 0$), so assume
$0 < b < a$.   
Now note the above sum is negative if and only
if
\begin{equation} \label{Fmtilde-abc}
\sum_{j=0}^b \frac{(a-j)(b-j)}{c-j} < \sum_{j=b}^a \frac{(a-j)(j-b)}{c-j}.
\end{equation}
This is certainly the case if
\[ b(b+1)(3a-b+1) =  \sum_{j=0}^b {(a-j)(b-j)} \le \sum_{j=b}^a {(a-j)(j-b)} =(a-b)(a-b+1)(a-b-1). \]
Writing $d=a-b$, we see this is true if
\[ b(b+1)(3d+2b+1) \le  d^3-d, \]
which holds if $b \le \frac d2 - 1$, i.e., if $b \le \frac{a-2}3$.

Similarly, the sum in the lemma is positive if
\[  \frac 1c  \sum_{j=0}^b {(a-j)(b-j)} \ge \frac 1{c-a} \sum_{j=b}^a {(a-j)(j-b)}, \]
i.e., 
if 
\[ b(b+1)(3d+2b+1) \ge  \frac{c}{c-a} \left( d^3-d \right). \]
Suppose $b \ge \frac a3$, i.e., $b \ge \frac d2$.  Then the above inequality is satisfied if
\[ \frac d2(\frac d2 +1)(4d+1) \ge  \left( 1 + \frac a{c-a} \right) \left( d^3-d \right), \]
which holds if
\[ \frac {d^2}4 (4d+1) \ge \left( 1 + \frac a{c-a} \right) d^3, \]
which holds if 
\[ \frac{c-a}{4a} \ge a \ge a-b = d. \]
This holds if $c \ge 5a^2 \ge 4a^2+a$.
\end{proof}

Now applying the first part of this lemma with $a=m$, $c=n$ and $b = \lceil 4b(n)-m-1 \rceil \le 4b(n) - m$, we see
\[ m \ge 3 b(n)  + \tfrac{1}{2} \implies \~F(m) < \~F(m+1), \]
hence $\~m_\opt(n) < 3b(n) + \frac 12$.

Similarly, applying the second part of the lemma with $a=m$, $c=n$ and 
$b = \lfloor 4b(n) - m - 1\rfloor \ge 4b(n) - 2$
we see
\[ m < 3b(n) - \tfrac{3}{2} \text{ and } n \ge 5m^2 \implies \~F(m) > \~F(m+1). \]
Note $m < 3b(n) - \frac{3}{2}$ implies $5m^2 \le n$ when $45(b(n)-1.5)^2 \le n$.  If $n$
is large so that $45(b(n)-1.5)^2 \le n$, then we have $\~m_\opt(n) \ge 3b(n) - \frac 32$.
This is satisfied if $n > 4050$.  
When $n \le 4050$, one can compute directly that \eqref{Fmtilde-ineq2} holds
for all $m < 3b(n) - \frac 32$.
 Lastly, note that $n > 3b(n)-\frac 32$ for $n \ge 5$.
\end{proof}

Hence, we should have $3b(n) - \frac 32 \le m_\opt(n) < 4b(n)$.  Here only the lower bound is 
conjectural.  At least for $n$ small, the table above suggests $\~m_\opt(n)$ is to be a very good approximation
for $m_\opt(n)$.  This  suggests
 that $m_\opt(n)$ grows like $\~m_\opt(n)$ plus some term of
smaller order, and we are led to

\begin{conj} \label{m4-conj2} As $n \to \infty$, we have the asymptotic
$m_\opt(n) \sim 3b(n).$
\end{conj}

\subsection*{Remarks}\hfill

\begin{invisimark} \label{never-cleanup-rem}
From Table \ref{mopt-tab}, we note that for $n \le 8$, one should not
clean up until all the objects are in the pile.  One might ask for $n \le 8$ 
if one should ever clean up, i.e.,
is $F(n)$ at least less than the cost of an average pile search, $\frac {n+1}2$?  Calculations
show this only true for $n=7$ and $n=8$, i.e., one should {\em never} clean up if $n \le 6$.
\end{invisimark}

\begin{invisimark} \label{M3-calc-rem} For $\mathcal M_3$, calculations also say that $m_\opt(n) = n$ for $n \le 8$,
but here it is only better to never clean up if $n \le 2$.  Furthermore, 
calculations suggest that $m_\opt(n) \sim 4b(n)$  for $\mathcal M_3$ . 
\end{invisimark}

\begin{invisimark} \label{savings-rem}
To see how close $\~F(m)$ is to $F(m)$, we plotted both for $n=20$ in Figure \ref{fig1}.
Note that there is significant cost savings to be had by waiting until $m_\opt$ to clean up.  However,
for large $n$, the graph of $F(m)$ will be more skewed as $F(n)$ is much larger than $F(1)$.
It is not feasible to compute all values of $F(m)$ for some large $n$, but we graph $\~F(m)$
for $n=100$ in Figure \ref{fig2}.  We expect that the graph of $F(m)$ will have a similar shape,
and that for large $n$, the cost savings of waiting until $m_\opt$ to clean up is proportionally 
smaller.  However, the cost savings should be more pronounced for non-uniform distributions.
\end{invisimark}

\begin{center}
\begin{figure}
\begin{tikzpicture}
	\begin{axis}[
		xlabel=$m$,
		ylabel=$$,
		legend entries = {$F(m)$, $\~F(m)$},
		legend style = {legend pos = north west},
		]
	\addplot[color=red, mark=*] coordinates{
( 1 ,  7.22386658784 )
( 2 ,  7.11746171245 )
( 3 ,  6.99458276474 )
( 4 ,  6.87412332227 )
( 5 ,  6.76290728722 )
( 6 ,  6.66429243642 )
( 7 ,  6.58025242957 )
( 8 ,  6.51148889616 )
( 9 ,  6.45548860194 )
( 10 ,  6.40141803534 )
( 11 ,  6.42438638139 )
( 12 ,  6.43856591688 )
( 13 ,  6.47837320768 )
( 14 ,  6.54701810181 )
( 15 ,  6.6488346379 )
(16, 6.79004794699)
(17, 6.98028331989)
(18, 7.23614856613)
(19, 7.591830101)
(20, 8.14481872387)
	};
	\addplot[color=blue, mark=*] coordinates {
( 1 ,  7.22386658784 )
( 2 ,  7.06428026507 )
( 3 ,  6.91930417664 )
( 4 ,  6.7894603062 )
( 5 ,  6.67533965227 )
( 6 ,  6.57761715255 )
( 7 ,  6.49707119509 )
( 8 ,  6.43460959719 )
( 9 ,  6.39130492667 )
( 10 ,  6.36844369696 )
( 11 ,  6.36759683884 )
( 12 ,  6.39072405851 )
( 13 ,  6.44033465613 )
( 14 ,  6.51974771885 )
( 15 ,  6.63353955542 )
( 16 ,  6.78837608417 )
( 17 ,  6.9947340548 )
( 18 ,  7.27103874741 )
( 19 ,  7.65624823248 )
( 20 ,  8.26911778588 )
	};
		\end{axis}
\end{tikzpicture}
\caption{Comparing $\~F(m)$ with $F(m)$ for $n=20$}
\label{fig1}
\end{figure}
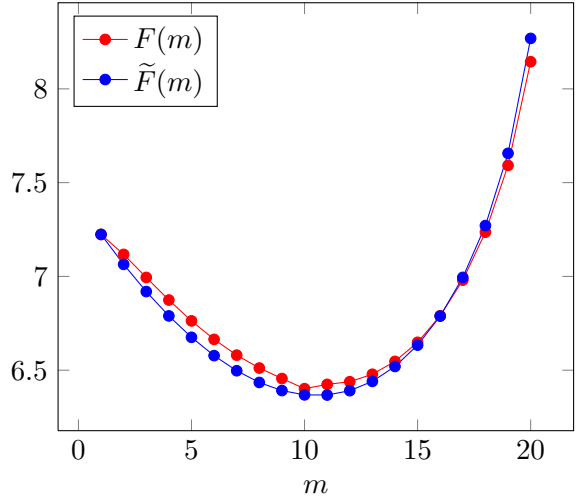
\end{center}

\begin{center}
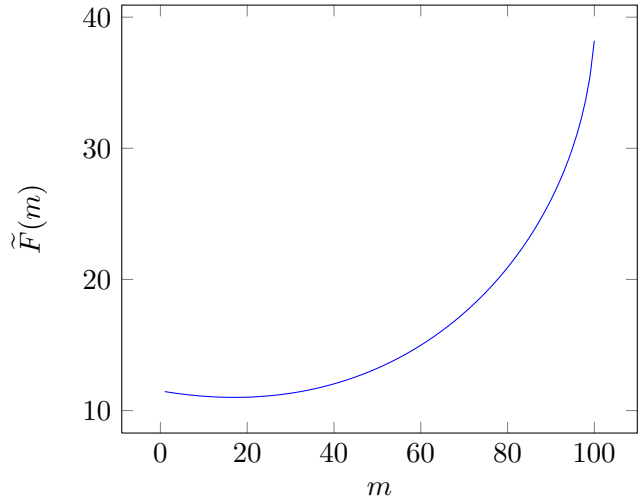
\begin{figure}
\begin{tikzpicture}
	\begin{axis}[
		xlabel=$m$,
		ylabel=\text{$\~F(m)$}]
	\addplot[color=blue] coordinates {
( 1 ,  11.4495871952 )
( 2 ,  11.3970767067 )
( 3 ,  11.3477547413 )
( 4 ,  11.3016446176 )
( 5 ,  11.2587701967 )
( 6 ,  11.2191559009 )
( 7 ,  11.1828267338 )
( 8 ,  11.1498083014 )
( 9 ,  11.1201268339 )
( 10 ,  11.0938092089 )
( 11 ,  11.0708829756 )
( 12 ,  11.0513763802 )
( 13 ,  11.0353183926 )
( 14 ,  11.0227387342 )
( 15 ,  11.0136679079 )
( 16 ,  11.0081372284 )
( 17 ,  11.0061788557 )
( 18 ,  11.0078258286 )
( 19 ,  11.0131121019 )
( 20 ,  11.0220725837 )
( 21 ,  11.0347431766 )
( 22 ,  11.0511608197 )
( 23 ,  11.0713635342 )
( 24 ,  11.0953904706 )
( 25 ,  11.1232819595 )
( 26 ,  11.1550795649 )
( 27 ,  11.1908261408 )
( 28 ,  11.2305658915 )
( 29 ,  11.2743444354 )
( 30 ,  11.3222088727 )
( 31 ,  11.3742078578 )
( 32 ,  11.4303916759 )
( 33 ,  11.4908123251 )
( 34 ,  11.5555236034 )
( 35 ,  11.6245812024 )
( 36 ,  11.6980428065 )
( 37 ,  11.7759681999 )
( 38 ,  11.8584193802 )
( 39 ,  11.9454606816 )
( 40 ,  12.0371589055 )
( 41 ,  12.1335834622 )
( 42 ,  12.2348065225 )
( 43 ,  12.3409031815 )
( 44 ,  12.451951635 )
( 45 ,  12.5680333703 )
( 46 ,  12.6892333721 )
( 47 ,  12.8156403459 )
( 48 ,  12.9473469604 )
( 49 ,  13.0844501106 )
( 50 ,  13.2270512038 )
( 51 ,  13.3752564721 )
( 52 ,  13.5291773131 )
( 53 ,  13.6889306622 )
( 54 ,  13.8546394014 )
( 55 ,  14.0264328071 )
( 56 ,  14.2044470442 )
( 57 ,  14.3888257087 )
( 58 ,  14.5797204296 )
( 59 ,  14.7772915333 )
( 60 ,  14.9817087827 )
( 61 ,  15.1931521989 )
( 62 ,  15.4118129785 )
( 63 ,  15.6378945193 )
( 64 ,  15.8716135725 )
( 65 ,  16.1132015385 )
( 66 ,  16.3629059311 )
( 67 ,  16.6209920361 )
( 68 ,  16.8877447977 )
( 69 ,  17.1634709721 )
( 70 ,  17.4485015962 )
( 71 ,  17.7431948288 )
( 72 ,  18.0479392373 )
( 73 ,  18.3631576158 )
( 74 ,  18.6893114469 )
( 75 ,  19.0269061415 )
( 76 ,  19.3764972321 )
( 77 ,  19.7386977384 )
( 78 ,  20.1141869897 )
( 79 ,  20.5037212714 )
( 80 ,  20.908146778 )
( 81 ,  21.3284155152 )
( 82 ,  21.7656050115 )
( 83 ,  22.2209430181 )
( 84 ,  22.6958388268 )
( 85 ,  23.1919235055 )
( 86 ,  23.7111023556 )
( 87 ,  24.2556244417 )
( 88 ,  24.8281764895 )
( 89 ,  25.4320124265 )
( 90 ,  26.0711365465 )
( 91 ,  26.7505700282 )
( 92 ,  27.4767521009 )
( 93 ,  28.2581689204 )
( 94 ,  29.1063896342 )
( 95 ,  30.0378828994 )
( 96 ,  31.0774683887 )
( 97 ,  32.2656247026 )
( 98 ,  33.6765518307 )
( 99 ,  35.4751047256 )
( 100 ,  38.2008348461 )
	};
	\end{axis}
\end{tikzpicture}
\caption{$\~F(m)$ for $n=100$}
\label{fig2}
\end{figure}
\end{center}

%
%
%
%

\section{Non-uniform distributions} \label{nonunif-sec}

%
%
%
%

Finally we comment on the problem for general probability distributions on
$\bf X$.  
Now if one defines the cost functions $S(X; \mathcal P)$ and $C(\mathcal P)$ using
 as algorithm $\bf A$ as in Section \ref{scf-sec}, these cost functions do
  not just depend upon the multiset of probabilities $\{ \mu(X_i) \}$, but upon the 
  specific distribution.

\begin{ex} \label{nuf-ex}
Fix $1 \le r \le n$ and $0 \le \epsilon \le 1$.  Now take the distribution given by 
$\mu(X_r) = 1-\epsilon$ and $\mu(X_i) = \epsilon/(n-1)$.  Assuming $\epsilon$ is small, then
most of the time one will be searching for $X_r$.  Depending on what $r$ is, the search (as well as cleanup) cost associated to $X_r$ might be as low as 1 or as high as $b_f(n) \approx \log_2(n)$.  
Hence, at least for certain values of $\epsilon$ and $n$, one might expect the 
answer to the associated search with cleanup optimization problem depends upon the choice
of $r$.
\end{ex}

Therefore, we define our cost functions not using the exact search costs given by 
algorithm $\bf A$, but rather on the associated average search costs.  Specifically,
in the complete memory case, we set
\begin{equation} \label{SXP-nuf}
 S(X; \mathcal P) = \begin{cases}
 s_{\mathcal L}(n-|\mathcal P|) & X \not \in \mathcal P \\
 s_{\mathcal P}(|\mathcal P|) & X \in \mathcal P
 \end{cases}
\end{equation}
and 
\begin{equation} \label{CP-nuf}
 C(\mathcal P) = \sum_{j=1}^{|\mathcal P|} s_{\mathcal L}(n-j).
\end{equation}  
In the case of the uniform distribution on $\bf X$,
this gives us the same optimization problem we studied above.  

Note that in the case of no memory, it may be better to always search the pile first,
depending on how skewed the distribution is.
For instance, in Example \ref{nuf-ex}, if
$\epsilon$ is sufficiently small, then with high probability at any $t \ge 1$, we will be looking for
$X_r$ and it will be in the pile.  Thus we should always search the pile first.  Furthermore,
by this reasoning (in either the complete or no memory case), for $\epsilon$ small enough,
we should clean up whenever another object gets in the pile, i.e.,  $m_\opt(n)=2$.

Consequently, we can decompose the average total cost as in the uniform case
\begin{equation}
 F(m) = m E_m \left[ \frac 1 \ell \right]  \cdot 2b(n) + 
 \sum_{j=1}^{m-1} E_m \left[ \frac {\tau_j}\ell \right] s_{\mathcal P}(j),
\end{equation}
though now the quantities $E_m \left[ \frac 1 \ell \right]$ and
$E_m \left[ \frac {\tau_j}\ell \right]$ will be more complicated.  In this case, the probability
functions for the underlying Markov process will follow more general sequential occupancy
distributions (see, e.g., \cite{urns} or \cite{charalam}).  

Note that for a nonuniform distribution, typically objects with higher probabilities will be in
the pile at any given time, so the pile search costs will be higher than in the uniform case.  Put 
another way, the expected waiting time $E_m [\ell]$ until cleanup is minimized
for the uniform distribution (see, e.g., \cite{nath2}, \cite{FGT}, \cite{BP} and \cite{BS} for results on $E_m[\ell]$).  
Therefore, the more skewed the
distribution is, the faster the probabilities $E_m\left[ \frac{\tau_j}\ell \right]$ should be increasing
in $m$, i.e., the smaller $m_\opt(n)$ should be, as indicated in our example above.
In particular, we expect $m_\opt(n)$ is maximized for the uniform distribution.

\newpage

%
%
%
%

\begin{appendix}
\section{Notation guide}
\label{app}

\subsubsection*{Section \ref{problem-statement-sec}}\hfill

\begin{tabular}{ll}
$\mathbf X$ & a set of $n$ objects (the books) $X_1, \ldots, X_n$ \\
$\mu$ & a probability measure on $\mathbf X$ (and later $\mathcal X_m$) \\
$\mathcal L$ & a sorted list (the shelves) \\
$\mathcal P$ & an unsorted list (the pile) \\
$\mathcal X_m = \mathcal X_{m,n}$ & the finite sequences (paths) of objects in $\mathbf X$ consisting of
  $m$ distinct \\ & objects, where the last object is distinct from the previous ones \\
$\chi$ & a path in $\mathcal X_m$ \\
$\chi_t$ & the $t$-th object in $\chi$ \\
$\ell(\chi)$ & the length of $\chi$ \\
$\mathcal P_\chi(t)$ & the set of objects in $\mathcal P$ at time $t$ along path $\chi$ \\
$S(X; \mathcal P)$ & the search cost for object $X \in \mathcal L \sqcup \mathcal P$ given a certain pile $\mathcal P$ \\
$C(\mathcal P)$ & the cleanup cost for a certain pile $\mathcal P$ \\
$S(\chi)$ & the total search cost along path $\chi$ \\
$C(\chi)$ & the cleanup cost for path $\chi$ \\
$F(m)=F(m;n)$ & the average total per-search cost for cleaning up when $|\mathcal P|=m$  \\
$m_\opt(n) = m_\opt(n;\mathcal M)$ & the argument which minimizes $F(m)$
\end{tabular}

\subsubsection*{Section \ref{scf-sec}}\hfill

\begin{tabular}{ll}
$\mathcal M_1$ & the no memory, unnumbered shelves model  \\
$\mathcal M_2$ & the no memory, numbered shelves model  \\
$\mathcal M_3$ & the complete memory, unnumbered shelves model  \\
$\mathcal M_4$ & the complete memory, numbered shelves model  \\
$\mathbf A$ & a search algorithm for the model \\
$b(j)$ & the average case successful binary search cost on a sorted list of length $j$ \\
$b_f(j)$ & the average case failed binary search cost on a sorted list of length $j$\\
$s(j)$ & the average case sequential search cost on a list of length $j$ \\
$s_{\mathcal L}(j)$ & the average cost to search for an element of $\mathcal L$ when the list size is $j$ \\
$s_{\mathcal P}(j)$ & the average cost to search for an element of $\mathcal P$ when the pile size is $j$ \\
$C_m$ & the average cleanup cost for a pile of size $m$  \\
\end{tabular}

\subsubsection*{Section \ref{esc-sec}}\hfill

\begin{tabular}{ll}
$E[f] = E_m[f] = E_{m,n}[f]$ & the expected value of a function on $\mathcal X_{m,n}$ \\
$\mathcal X_m^{(\ell)}$ & the paths in $\mathcal X_m$ of length $\ell$ \\
$F_S(m)$ & the expected search cost per search \\
$\ssk{a}{b}$ & the Stirling number of the second kind \\
$S_{\mathcal L}(\chi)$ & the contribution to $S(\chi)$ from searches for
objects in $\mathcal L$ \\
$S_{\mathcal P}(\chi)$ & the contribution to $S(\chi)$ from searches for objects in $\mathcal P$ \\
$Sym(\mathbf X)$ & the symmetric group on $\bf X$ \\
$\tau_j(\chi)$ & the number of times one does $j$-element pile search along $\chi$ \\
\end{tabular}

\subsubsection*{Section \ref{ecc-sec}}\hfill
\smallskip

\begin{tabular}{ll}
$F_{\mathcal L}(m)$ & the expected value of $S_{\mathcal L}$ per search \\
$F_{\mathcal P}(m)$ & the expected value of $S_{\mathcal P}$ per search \\
$F_C(m)$ & the expected cleanup cost per search
\end{tabular}



\subsubsection*{Section \ref{approx-sec}}\hfill
\smallskip

\begin{tabular}{ll}
$\~F(m)$, $\~F_{\mathcal L}(m)$, $\~ F_{\mathcal P}(m)$ &  
certain approximations for $F(m)$, $F_{\mathcal L}(m)$, $F_{\mathcal P}(m)$ \\ 
$\~m_\opt(n)$ & the argument minimizing $\~F(m)$
\end{tabular}

\end{appendix}

%
%
%
%

%
%
%
%

\end{document}